\documentclass[11pt]{article}                     
\normalbaselines
\usepackage{mathptmx}      
\usepackage{amssymb,latexsym,amsfonts,amsmath,amsthm}
\newtheorem{theorem}{\bf Theorem
}
\newtheorem{assumption}{\bf Assumption}
\newtheorem{lemma}{\bf Lemma}

\newcommand{\N}{\mathbb{N}}
\newcommand{\Z}{\mathbb{Z}}

\newcommand{\fff}{\mathcal{F}}
\renewcommand{\P}{\mathbb{P}}
\DeclareMathOperator{\E}{{\mathbb E}}
\DeclareMathOperator{\one}{{ 1\hspace*{-0.55ex}I}}
\newcommand{\cond}{\hspace*{1ex} \rule[-1ex]{0.15ex}{3ex}\hspace*{1ex}}

\usepackage{harvard}
\newcommand{\citet}{\citeasnoun}

\begin{document}


\author{Peter Jagers\addtocounter{footnote}{-2}\thanks{Department of Mathematical Sciences, 
Chalmers University of Technology 
and University of Gothenburg,
SE-412 96 Gothenburg, Sweden.
Emails: \texttt{jagers@chalmers.se},\ \texttt{sergei.zuyev@chalmers.se}} \and \addtocounter{footnote}{-1}
Sergei Zuyev\footnotemark}

\title{Populations in environments with
  a soft carrying capacity are eventually extinct \addtocounter{footnote}{1}\thanks{To appear in J.Math.Biol.}
}

\date{\today}

\maketitle

\begin{abstract}
Consider a population whose size changes stepwise by its members reproducing
or dying (disappearing), but is otherwise quite general.  Denote the
initial (non-random) size by $Z_0$ and the size of the $n$th change by
$C_n$, $n= 1, 2, \ldots$.  Population sizes hence develop successively
as $Z_1=Z_0+C_1,\ Z_2=Z_1+C_2$ and so on, indefinitely or until there
are no further size changes, due to extinction. Extinction is thus
assumed final, so that $Z_n=0$ implies that $Z_{n+1}=0$, without there
being any other finite absorbing class of population sizes. We make no
assumptions about the time durations between the successive
changes. In the real world, or more specific models, those may be of
varying length, depending upon individual life span distributions and
their interdependencies, the age-distribution at hand and intervening
circumstances. We could consider toy models of Galton-Watson type
generation counting or of the birth-and-death type, with one
individual acting per change, until extinction, or the most general
multitype CMJ branching processes with, say, population size
dependence of reproduction. Changes may have quite varying
distributions. The basic assumption is that there is a {\em carrying
capacity}, i.e. a non-negative number $K$ such that the conditional
expectation of the change, given the complete past history, is
non-positive whenever the population exceeds the carrying capacity.
Further, to avoid unnecessary technicalities, we assume that the
change $C_n$ equals -1 (one individual dying) with a conditional
(given the past) probability uniformly bounded away from 0. It is a
simple and not very restrictive way to avoid parity phenomena, it is
related to irreducibility in Markov settings.  The straightforward,
but in contents and implications far-reaching, consequence is that all
such populations must die out. Mathematically, it follows by a
submartingale convergence property and positive probability of
reaching the absorbing extinction state.

  \medskip
\noindent  \textsc{Keywords:} population dynamics, extinction, martingales, stochastic stability

\noindent  \textsc{AMS 2010 Mathematics Subject Classification:} 92D25, 60G42, 60K40

\end{abstract}

\maketitle

\section{``All surnames tend to be lost''}

Almost a century and a half have passed since \citet{Galton} and
\citet{GW} introduced their famous simple branching process followed
by the infamous conclusion that all families (``surnames'') must die
out: ``All surnames tend to extinction [...] and this result
might have been anticipated, for a surname lost can never be
recovered.''  Since long it is textbook knowledge, that the extinction
probability of supercritical Galton-Watson (and more general)
branching processes is less than one, the alternative to extinction
being unbounded exponential growth. For a loose discussion of this
dichotomy and reflections on what circumstances that might salvage
Galton's and Watson's conclusion, see, e.g., \cite{Haccou}. Here
we prove almost sure extinction of quite general, stepwise changing
populations, which can reach any size but live in a habitat with a
carrying capacity, interpreted as a border line where reproduction
becomes sub-critical, but which may be crossed by
population size, i.e.\ a {\em soft}, carrying capacity.

Mathematically, what happens is that the population size process
becomes a super-martingale, when
crossing the carrying capacity, and extinction follows from a
combination of martingale properties. In the population dynamics
context, the result is fundamental and applies broadly, e.g., to
Markov and to general population size-dependent branching processes as
discussed by \citet{JagKle11}, provided the conditional survival
times and reproduction processes, given the past satisfy continuity
and conditional independence conditions.

The concept of a soft carrying capacity, strictly defined in
Assumption 1 beneath, is new but not unrelated to earlier discussion
in biological and mathematical population dynamics on ideas of
density dependence, see, e.g., papers by \citet{Gin:90}, \citet{Ber:91}
and \citet{NisBen:89}.

\section{Dynamics of Population Changes}

Consider a population which starts from a non-random number $Z_0$ of
individuals. These can be of various types and ages, we shall not go
into details. Changes occur successively by the death or reproduction
of the population members, and are denoted $C_n,\ n\in \N$, where $\N$
stands for the set of positive integers, and $C$ for change.  Each size
change is thus an integer valued random variable. After the first
change, there are $Z_1=Z_0+C_1$ individuals present, and generally
$Z_{n+1} = Z_n +C_{n+1}$, as long as $Z_n>0$. If $Z_n=0$, then so is
$Z_{n+1}$. The population has died out. We do not make any assumptions
about the time between changes, which in real life or more detailed
models may be quite varying and influenced by many factors, external
or internal, like the population size or the age-distribution of
individuals in the population.  Nor are there any assumptions of
customary kind about the distributions of or interdependencies between
the various $C_n$'s. Without loss of generality, we may assume that
$\P(C_n=0\cond \fff_{n-1})=0$: one can always change the indices $n$
to correspond solely to non-zero changes in the population size, of course
with the corresponding change in the conditional distributions of the
size and the time to the next event.

Simultaneous deaths of a few individuals are not excluded, but this
must not always be the case: we later assume that with positive
probability only \emph{one} individual in the population is dying at a
given index $n$. This will always be satisfied by systems where,
somewhat vaguely, individual lifespans have jointly continuous
distributions and bearings occur in a point process with a finite
intensity.

As a somewhat more precise example, satisfying our requirements,
consider a general (CMJ) branching process inspired setup, where
individuals have independent identically distributed life spans with a
continuous distribution function. Assume that during life individuals
give birth according to a point process whose intensity may be both
population-size-dependent and influenced by maternal age. A classical
type of simple processes meeting such requirements are those of
birth-and-death processes.

Another interesting case is that of a cell population, where cells
evolve in cycles, completed cycles are ended by mitotic division.  The
cycles may be dependent, but with a positive probability only one cell
divides: the size change is either -1 (if the cell dies before
completing its cycle) or +1 (if two fresh cells replace the mother
after mitosis). A ``division'' resulting in just one daughter cell,
sometimes referred to as ``asymmetrical'', would thus have to be
interpreted either as a division closely followed by the death of one
of the daughters, or just the mother cell living on, i.e.\ no change
in numbers (which however would have repercussions on the assumptions
for life span distributions).

A toy model, inspired by the Galton-Watson or Moran process, would be to let
the changes $C_{n+1}$ occur at the real time points $ n=1, 2,\ldots$ by a
(somehow chosen) individual either dying or being replaced by two or
more individuals, according to a distribution that might depend upon
the population size $Z_n$.

\section{Carrying Capacities and Extinction}
We denote the sigma-algebra of all events up to and including the
$n$-th occurrence by $\mathcal{F}_n$, and introduce a {\em carrying
  capacity} $K>0$, thought of as a large natural number. Being a
carrying capacity of the population means that the conditional
expectation of the impending change, given its past, satisfies
\begin{assumption}\label{ass:cc}
  \begin{equation}          \label{eq:cc}
    \E[C_{n+1}|\fff_n] \leq 0, \quad \mathrm{if}\ \ Z_n \geq K.
  \end{equation}
\end{assumption}
Thus, the carrying capacity, as mentioned, does not provide a
categorical barrier: population size may exceed it but then it tends
to decrease. For individual based models with a carrying capacity,
see, e.g., \cite{Fan} or \cite{JagKle11}. 

The super-martingale property \eqref{eq:cc} of the population size process is one
basic leg of our analysis, the other being the fact that each
individual, whatever the circumstances,
always runs a definite risk of death unrelated to the others. Specifically,
denoting by $\Z_+$ the set of non-negative natural numbers, we make
\begin{assumption} \label{ass:death}
  There is an $\epsilon >0$ such that
  \begin{equation}\label{eq:death}
    \P(C_{n+1}=-1|\mathcal{F}_n) \geq \epsilon\ \mathrm{for\ all}\ n\in \Z_+\,.
  \end{equation}
\end{assumption}
This is a technical assumption which in many models can be
relaxed. Its purpose is to avoid traps when the system gets into a
subset of states not containing zero without possibility to leave
it. Or the parity phenomenon when, for instance, the initial number
$Z_0$ of individuals is odd but only the changes $C_n$ by even numbers
have non-zero probabilities: obviously, such a population will never get
extinct. In a Markovian setting, the assumption guarantees that the
chain is irreducible.

Define $\nu_1$ to be the first visit of the process
below the carrying capacity,
\[\nu_1 :=\inf \{n\in \Z_+;\ Z_n<K\},\]
 Hence, if $0<Z_0<K$,  $\nu_1=0$ and $Z_{\nu_1} =Z_0$ , whereas  
\[Z_0 \geq K\Rightarrow 1\leq \nu_1\leq \infty \mbox{ and } Z_{\nu_1}
  \leq K-1,\]
provided $\nu_1<\infty$. 
\begin{lemma}
$\{Z_{n\wedge \nu_1}\}$ is a non-negative supermartingale whose
expectation is bounded by $Z_0$. 
\end{lemma}

\begin{proof}
Since $\nu_1$ is a stopping time, it holds for any $n\in \Z_+$ that
\begin{multline*}
  \E[Z_{(n+1)\wedge\nu_1}|\mathcal{F}_n] =
  \E[Z_{(n+1)\wedge\nu_1}\one_{\nu_1 \leq n} |\mathcal{F}_n]+
  \E[Z_{(n+1)\wedge\nu_1}\one_{\nu_1 > n} |\mathcal{F}_n] \\
  =\E[Z_{\nu_1}\one_{\nu_1 \leq n}|\mathcal{F}_n] +
  \E[Z_{n+1}\one_{\nu_1 > n}|\mathcal{F}_n] \leq
  Z_{\nu_1}\one_{\nu_1 \leq n} +Z_n\one_{\nu_1 > n} =Z_{n\wedge\nu_1}.
\end{multline*}
\end{proof}

Hence, the process $(Z_{n\wedge \nu_1})$ converges almost surely (and
in $L^1$). Since further $|Z_{n+1}-Z_n|\geq 1$, then on the event
$\{\nu_1=\infty\}$ the sequence $(Z_{n\wedge \nu_1}) = (Z_n)$
diverges. Thus
\begin{displaymath}
\P(\nu_1=\infty) \le \P((Z_{n\wedge \nu_1})
\text{ does not converge})=0,
\end{displaymath}
and $Z_{n\wedge \nu_1}\to Z_{\nu_1}\leq Z_0 \wedge (K-1)$ a.s.

Continue to define $\mu_1:= \inf \{n>\nu_1;\ Z_n\geq K\}\leq \infty$,
and proceed recursively to
\begin{align*}
\nu_{k+1} & := \inf \{n>\mu_k;\ Z_n<K\},\ k\in \mathbb{N}, \\
\mu_k & := \inf \{n>\nu_k;\ Z_n\geq K\},\ k \in \mathbb{N},
\end{align*}
indefinitely or until one of the $\nu_k$ is infinity.
$\mathbb{N}$, as usual, stands for the set of natural numbers. 
Clearly, extinction $\{Z_n=0\}$ must occur after the last $\nu_k<\infty$, if there
is any.
\begin{theorem}\label{th:main}
  Under the two basic assumptions \eqref{eq:cc} and \eqref{eq:death}
  made, of a carrying capacity and a definite individual death risk,
  $\P(Z_n \to 0)=1$, i.e. extinction is (almost) certain.
\end{theorem}

\begin{proof}
If $\nu_k<\infty$ then so is $\mu_k$, unless the population dies out
before reaching or passing $K$. Denote $Z_{\nu_k}:=z_k$ for short. Then
\begin{displaymath}
  \P(\mu_k=\infty\cond\fff_{\nu_k})\geq \P(C_{\nu_k+1}=-1, C_{\nu_k+2}=-1,\dotsc,
  C_{\nu_k+z_k}=-1\cond \fff_{\nu_k}).
\end{displaymath}
Using \eqref{eq:death} and the tower property of conditional expectations,
\begin{multline*}
  \P(C_{\nu_k+1}=-1, C_{\nu_k+2}=-1\cond \fff_{\nu_k})=
  \E\big[\E[\one_{C_{\nu_k+1}=-1} \one_{C_{\nu_k+2}=-1}\cond
  \fff_{\nu_k+1}]\cond \fff_{\nu_k}\big]\\
  =\E\big[\one_{C_{\nu_k+1}=-1} \E[\one_{C_{\nu_k+2}=-1}\cond
  \fff_{\nu_k+1}]\cond \fff_{\nu_k}\big]\geq \epsilon \E[\one_{C_{\nu_k+1}=-1}\cond
  \fff_{\nu_k}]\geq \epsilon^2,
\end{multline*}
and so on, leading to
\begin{displaymath}
\P(\mu_{k}<\infty|\nu_{k}<\infty)\leq 1-\epsilon^{z_k} \geq p:=1-\epsilon^{K-1}
\end{displaymath}
because $z_k=Z_{\nu_k}\leq K-1$.
By the supermartingale property \eqref{eq:cc}, $(Z_n)$ must return (almost)
always below $K$ from a level equal to or above the carrying
capacity. Hence, almost surely 
\begin{displaymath}
\nu_{k+1}<\infty \Leftrightarrow \mu_{k}<\infty,\ k=1,2,\ldots
\end{displaymath}
Since the sequence $(\mu_k)$ does not decrease, it follows that 
\begin{multline*}
  \P(\mu_{k}<\infty)=\P(\mu_{k}<\infty\cond\mu_{k-1}<\infty)\,\P(\mu_{k-1}<\infty)=\\
  =\P(\mu_{k}<\infty\cond \nu_{k}<\infty)\,\P(\mu_{k-1}<\infty)\leq p\,
    \P(\mu_{k-1}<\infty)\leq \ldots \leq p^k\to 0.
  \end{multline*}

Hence,
\begin{displaymath}
\P(\exists k:\ \mu_{k}=\infty)=
  \lim_{k\to\infty}\P(\mu_{k}=\infty)=1.
\end{displaymath}

\end{proof}

Qualitatively, depending on the starting state, the population either
gets extinct quickly or evolves below and around the carrying capacity
$K$ until it eventually dies out. The population size, although
unbounded, does not get much larger than $K$. Indeed, from a
supermartingale form of Doob's maximal inequality, see, e.g.,
Cor.~2.4.6 in the book by \citet{MenPopWad:16},
\begin{displaymath}
  \P\big(\max_{n\geq 0} Z_{(\mu_{k-1}+n)\wedge \nu_k}\geq x\cond
  \fff_{\mu_{k-1}}\big)\leq \frac{K-1}{x}
\end{displaymath}
of course non-trivial only for $x\geq K$.

Although extinction is almost certain, the number of steps to it may,
however, be quite large. For instance, when $K$ is big and $Z_n$ is a
submartingale (i.e.\ $\E[C_{n+1}|\fff_n] \geq 0$) on the set $\{Z_n<K\}$,
the system of size $K-1$ needs to go a long way against a non-negative
drift to reach 0. Applying Doob's maximal inequality to the supermartingale
$X_n=K-Z_{(\nu_{k}+n)\wedge \mu_k}$ with $X_0=1$, we have that
\begin{displaymath}
  \P(Z_{\nu_k+n\wedge \mu_k}=0\cond \fff_{\nu_k})=\P\big(\max_{n\geq 0}
  X_n\geq K\cond \fff_{\nu_k}\big)\leq \frac{1}{K} 
\end{displaymath}
so it takes on average at least $K$ excursions to the domain below the
capacity to die out. In the general case we consider, nothing more can
be said: our model includes, as a particular example, the symmetric
simple random walk for which the maximal inequality is sharp. But
under additional assumptions, the average number of excursions and
time to extinction may grow exponentially in $K$ (cf.\ the exponential
lower bound on the extinction in the proof of Theorem~\ref{th:main}
above). For instance, this is the case when the increments $C_n$ are
totally bounded and the mean drift below $K$ is strictly positive:
$\E [C_{n+1}\cond\fff_n]\geq \delta$ almost surely for some $\delta>0$
and $0<Z_n<K$ , see, e.g., Th.~2.5.14 by
\citet{MenPopWad:16}. Similarly, in the presence of a strictly
negative drift above the carrying capacity,
$\E [C_{n+1}\cond\fff_n]\leq -\delta$ almost surely for $Z_n> K$, by
Th.~2.6.2 in the same book, we have that $\E [\nu_k]\leq K/\delta$ for
all $k\in\N$. If, in addition, $C_n$ are totally bounded then
according to Th.~2.5.14 there, the probability for the population to
reach size $K+x$ during an excursion above the capacity decays at
least exponentially in $x$. Qualitatively, in the presence of the
drifts uniformly separated from 0, the population size bounces around
the carrying capacity $K$ for quite a long (the time scales exponentially
with $K$) before eventually dying out.

Note also, that the uniform positivity condition in \eqref{eq:death}
is necessary for the imminent extinction: it is easy to produce
examples when $\P(C_{n+1}=-1\cond Z_n=1)$ decays with $n$ so quickly
that with positive probability the jump to 0 never happens although
the absorbing state $0$ remains attainable with positive probability.

\section*{Acknowledgement}
\label{sec:acknowledgement}
The authors are thankful to Mikhail Grinfeld for inspiring discussions
during his visit to Chalmers in the framework of the Strathclyde
University Global Engagements International Strategic Partnerships
Grant.


\end{document}